\documentclass[preprint]{elsarticle}
\usepackage{amsmath,amsthm,amssymb}
\usepackage{times}
\usepackage{enumerate}
\usepackage{url}
\usepackage{lineno}
\usepackage{microtype}
\usepackage[american]{babel}
\usepackage{indentfirst}
\usepackage{color}
\usepackage{geometry}
\usepackage{epsfig}
\newtheorem{result}{\textbf{Theorem}}
\newtheorem{proposition}{Proposition}[section]
\newtheorem{lemma}{Lemma}[section]
\newtheorem{definition}{Definition}[section]

\newtheorem{remark}{Remark}[section]
\newtheorem{example}{Example}[section]
\numberwithin{equation}{section}

\begin{document}
\baselineskip=15pt

\begin{frontmatter}   % 必须添加 frontmatter 环境

\title{Commutation Properties of Semi-groups for Contact Type Hamilton-Jacobi Equation}

\author{Guyu Jin}
\address{Graduate School of Mathematical Sciences, The University of Tokyo, 3 Chome-8-1 Komaba, Meguro City, Tokyo 153-8902, Japan}
\ead{guyu567@g.ecc.u-tokyo.ac.jp}   % 使用 \ead 命令代替 \email

 % 如果要填分类号，放在括号内

\begin{abstract}
We know that there exist semi-groups for contact type Hamilton-Jacobi equations, which refers to \cite{KLJ2}. Guy Barles and Agnès Tourin give a proof of the commutation properties for normal Hamilton-Jacobi equations at \cite{GA}. In this article, we provide a proof of the commutation property of semi-groups for contact type Hamilton-Jacobi equations.
\end{abstract}


\end{frontmatter}

\section{Introduction}

For the Cauchy problems for contact type Hamilton-Jacobi equations, which is defined for $x \in \mathbb{R}^n$ and $t \geq 0$, we have the equation in this form:
\begin{equation}\label{HJ}\tag{HJ}
 \left\{
   \begin{aligned}
   & u_t(x,t)+H\Big(x,Du(x,t),u(x,t)\Big)=0\quad \textrm{in}\quad \mathbb R^n\times(0,\infty)
   \\
   &u(x,0)=u_0(x),\quad u_0\in \text{Lip}(\mathbb R^n)\cap BUC(\mathbb R^n),
   \\
   \end{aligned}
   \right.
\end{equation}

we have the Lax-Oleinik semi-group defined as
\begin{align}
    S_H(t) [u_0(x)]:=u(x,t).
\end{align}

The commutation properties for general Hamilton-Jacobi equations have been proved by Guy Barles and Agnès Tourin. They gave the result that the commutation properties for the semi-groups is equivalent to the solvability of the associated multi-time Hamilton-Jacobi equation in \cite{GA}.  By giving the commutative Lax-Oleinik semi-groups, Davini and Zavidovique examined the weak KAM theoretic aspects of the commutation property and showed that the two Hamiltonians have the same weak KAM solutions, the same Aubry set, and the same Peierls barriers in \cite{davini2013weak} and \cite{zavidovique2010weak}. In \cite{zavidovique2010weak}, Zavidovique systematically investigates the weak KAM theory for commuting Tonelli Hamiltonians. He proves that if two Tonelli Hamiltonians $G$ and $H$ Poisson commute, then their Lax–Oleinik semi-groups commute.

Contact Hamiltonian systems also represent the odd-dimensional counterpart of Hamiltonian systems in the context of contact geometry. In Appendix 4 of \cite{arnold1989mathematical}, Arnold provides a detailed introduction to the important concepts in contact dynamics. For the relationship between contact structures and first-order partial differential equations in general form, see his another work \cite{arnold2004lectures}. For the recent main development of contact type Hamilton-Jacobi equation, we can see \cite{cannarsa2020herglotz} \cite{wang2019aubry} \cite{wang2021weak} \cite{ishii2022hamilton} \cite{wang2023time} \cite{wang2025time} \cite{ni2024nonlinear} \cite{ni2025representation}.

In contact geometry, the notion of observables is governed by the Jacobi bracket, a natural generalization of the Poisson bracket to odd-dimensional manifolds. Introduced by Kirillov and Lichnerowicz, the Jacobi bracket arises from the contact form \(\theta\) via the Reeb vector field \(\xi\) and a bivector field \(\Lambda\), encoding the Lie algebra structure of smooth functions on a contact manifold. In particular, the subspace \(C_b^\infty(M) = \{f \in C^\infty(M) \mid \xi \cdot f = 0\}\) forms a Poisson subalgebra, which plays a central role in the geometric quantization of contact manifolds. Functions in this subalgebra correspond to infinitesimal symmetries preserving the contact form and serve as the classical observables that admit a consistent quantum representation. More details on Jacobi brackets are referred to \cite{SF}. And different from Zavidovique in \cite{zavidovique2010weak}, we will consider Jacobi commutation instead of Poisson commutation.

In this passage, we focus on Jacobi commutation and the contact type first-order Hamilton-Jacobi equations. To the best of our knowledge, this is the first work that systematically investigates commutation properties of contact type Hamilton–Jacobi equations from the point of view of Jacobi brackets. We will show the commutation properties for the semi-groups equipped with different contact Hamiltonians first. 

We will give the conditions for $C^1$ Hamiltonian $H$, $F$: $\mathbb R^n \times \mathbb R^n \times \mathbb R \mapsto \mathbb R$. $H$, $F$ satisfies 
\begin{itemize}
\item[(H1)] For each $R>0$, $\exists K_R >0$ s.t. $|H(x,p,u)| \leq K_R$ and $|D_pH(x,p,u) |\leq K_R(1+|x|)$ in $\mathbb R^n \times B(0,R) \times \mathbb R$ 
\item[(H2)] $\lim_{|p|\to+\infty}\inf_{x\in\mathbb R^n ,u \in \mathbb R}H(x,p,u)/|p|=+\infty$.
\item[(H3)] $p\mapsto H(x,p,u)$ is convex for each $(x,u)\in \mathbb R^n\times \mathbb R$.
\item[(H4)] $H(x,p,u)$ is uniformly Lipschitz in $u$, i.e. $ \exists C >0, s.t.|\frac{\partial H}{\partial u}(x,p,u)| \leq C$ for each $(x,p,u) \in \mathbb R^n\times \mathbb R^n\times \mathbb R$
\item[(H5)] $H$ and $F$ satisfies 
     \begin{align*}
          D_x H(x,p,u) D_pF(x,p,u)- D_p H(x,p,u) D_xF(x,p,u) + p \cdot D_pF(x,p,u) \cdot \frac{\partial H}{\partial u} 
          \\ - p \cdot D_pH(x,p,u) \cdot \frac{\partial F}{\partial u} + \frac{\partial F}{\partial u}(x,p,u) H(x,p,u)-\frac{\partial H}{\partial u}(x,p,u) F(x,p,u)=0,
     \end{align*}
     for every $(x,p,u) \in \mathbb R^n \times \mathbb R^n \times \mathbb R$
\end{itemize}

We first give the definition of the viscosity solution to (HJ).

\begin{definition}
    \item[(i)] We say that $u$ is a viscosity subsolution to (HJ), if for each $\phi \in C^1(\mathbb R^n \times (0,\infty))$ such that $u(x_0,t_0)=\phi(x_0,t_0)$, and $u-\phi$ attains its strict maximum at $(x_0,t_0) \in \mathbb R^n \times (0,\infty)$, then 
          \begin{align*}
              \phi_t(x_0,t_0)+H\Big(x,D\phi(x_0,t_0),\phi(x_0,t_0)\Big) \leq 0
          \end{align*}
    and $u(\cdot ,0) \leq u_0$;
    \item[(ii)] We say that $u$ is a viscosity supersolution to (HJ), if for each $\phi \in C^1(\mathbb R^n \times (0,\infty))$ such that $u(x_0,t_0)=\phi(x_0,t_0)$, and $u-\phi$ attains its strict minimum at $(x_0,t_0) \in \mathbb R^n \times (0,\infty)$, then 
          \begin{align*}
              \phi_t(x_0,t_0)+H\Big(x,D\phi(x_0,t_0),\phi(x_0,t_))\Big) \geq 0
          \end{align*}
    and $u(\cdot ,0) \geq u_0$;
    \item[(iii)] We say that $u$ is a viscosity solution to (HJ), if $u$ is both the viscosity subsolution and viscosity supersolution to (HJ).
\end{definition}

We have some propositions for the viscosity solutions.

\begin{proposition}
    Assume that $H$ satisfies (H1)-(H4), then there exists the unique Lipschitz viscosity solution of (HJ).
\end{proposition}
    
This proposition referred to \cite{ishii2022hamilton} Theorem A.1. We will give the proof when $M$ is replaced by $\mathbb{R}^n$.

\begin{proof}
   We may assume that $T < \infty$. Choose a constant $C_1>0$ such that
   \begin{equation}
       |H(x, Du_0(x), u_0(x))| + |Du_0(x)|^2 \leq C_1 \quad \text{for all } x \in \mathbb R^n.
   \end{equation}
   Define the functions $f^\pm \in C^1(\mathbb R^n \times [0, T], \mathbb{R})$,by
   \begin{equation}
       f^+(x, t) = u_0(x) + C_1 \Lambda^{-1}(e^{\Lambda t} - 1) \quad \text{and } f^-(x, t) = u_0(x) - C_1 \Lambda^{-1}(e^{\Lambda t} - 1).
   \end{equation}
   Choose a constant \( C_2 > 0 \) so that \(-C_2 \leq f^- \leq f^+ \leq C_2\) on \( \mathbb R^n \times [0, T] \). 
   We define the new Hamiltonians $\tilde{H}, \tilde{H}_k, \hat{H}_k$, by setting
   \begin{align}
       \tilde{H}(x, p, u) = H(x, p, \max\{-C_2, \min\{C_2, u\}\}) \quad \text{for } (x, p, u) \in \mathbb R^n \times \mathbb R^n \times \mathbb{R},\\
       \tilde{H}_k(x, p, u) =\min\{k, \tilde{H}(x, p, u)\} \quad \text{for } (x, p, u) \in  \mathbb R^n \times \mathbb R^n \times \mathbb{R}, \quad k \in \mathbb{N}\\
       \hat{H}_k(x, p, u) = \tilde{H}_k(x, p, u) + \frac{1}{k} |p|^2 \quad \text{for } (x, p, u) \in  \mathbb R^n \times \mathbb R^n \times \mathbb{R}, \quad k \in \mathbb{N}.
   \end{align}
   It is easily seen that the functions $\tilde{H}(x, p, u), \tilde{H}_k(x, p, u), \hat{H}_k(x, p, u)$ are Lipschitz continuous in $u$, with $\Lambda$ as a Lipschitz bound, that for any $f \in C^1(\mathbb R^n \times (0, T), \mathbb{R})$if \( |f| \leq C_2 \) on \(\mathbb R^n  \times (0, T) \), then,
\[
\tilde{H}(x, D f(x, t), f(x, t)) = H(x, D f(x, t), f(x, t)) \quad \text{for all } (x, t) \in \mathbb R^n \times (0, T),
\]
$|\tilde{H}_k(x, Du_0(x),u_0(x))| \leq |\tilde{H}(x, Du_0(x), u_0(x))|$ and that \( |\hat{H}_k(x, D u_0(x), u_0(x))| \leq C_1 \) for all \( x \in \mathbb R^n \) and \( k \in \mathbb{N} \).
Compute that

\[f_t^+ + H(x, Df^+, f^+) \geq C_1 e^{\Lambda t} + \hat{H}_k(x, Du_0, u_0) - C_1(e^{\Lambda t} - 1) \geq 0,\]

to see that for any \( k \in \mathbb{N} \), \( f^+ \) is a classical supersolution of
\begin{equation}
    u_t + \hat{H}_k(x, D u, u) = 0 \quad \text{in } \mathbb R^n  \times (0, T).\label{Hk}
\end{equation}

Similarly, we find that for any \( k \in \mathbb{N} \), \( f^- \) is a classical subsolution of (\ref{Hk}) . With \cite{ishii2022hamilton} Lemma A.2 and some modifications on the proof, we have the comparison principle for this type of equation. Moreover, note that \( f^-(x, 0) = f^+(x, 0) = u_0(x) \) for all \( x \in M \) and \( f^- \leq f^+ \) on \( \mathbb R^n \times [0, T] \). The Perron method yields a Crandall-Lions viscosity solution of (\ref{Hk}). That is,
\[u^k(x, t) = \sup\{f(x, t) : f  \text{ is a viscosity subsolution of ($\ref{Hk}$)}, f^- \leq f \leq f^+\text{ on }  \mathbb R^n \times (0, T)\}\]
Since \( f^\pm(x, 0) = u_0(x) \) for all \( x \in  \mathbb R^n \), we may extend the domain of \( u^k \) to \(  \mathbb R^n \times [0, T] \).
We note that, thanks to (H2), \(\tilde{H}_k(x,p,u) = k\) if \(|p|\) is sufficiently large and, hence, the function \(\tilde{H}_k\) is bounded and uniformly continuous on \(\mathbb R^n \times \mathbb R^n \times \mathbb{R}\). Since the upper and lower semi-continuous envelopes \((u_k)^*\) and \((u_k)_*\) are, respectively, a viscosity subsolution and supersolution of (\ref{Hk}), we find by the comparison principle  that \((u_k)^* \leq (u_k)_*\) on \(\mathbb R^n \times [0,T)\), which implies that \(u_k \in C(\mathbb R^n \times [0,T),\mathbb{R})\).

We now show that the family \(\{u^k\}_{k \in \mathbb{N}}\) is equi-Lipschitz continuous on \(\mathbb R^n \times (0,T)\). For this, we show first that the functions \(\tilde{H}_k\), with \(k \in \mathbb{N}\), are coercive uniformly in \(k\). That is, for any \(R > 0\) there exists \(Q > 0\), chosen independently of \(k\), such that for any \(k \in \mathbb{N}\),
\begin{equation}
    \hat{H}_k(x,p,u) > R \quad \text{if } |p| > Q.\label{10}
\end{equation}
Indeed, when \(R > 0\) is fixed, by (H2) we may choose \(Q \geq R\) so that \(\tilde{H}(x,p,u) > R\) if \(|p| > Q\). Using the inequality
\[R < k + \frac{1}{k}R^2,\]

we find that if \(|p| > Q\), then
\[\hat{H}_k(x,p,u) \geq \min\left\{k + \frac{1}{k}Q^2, \tilde{H}(x,p,u)\right\} > R.\]
Fix \(h > 0\) sufficiently small. Since \(u^k \geq f^{-}\) on \(\mathbb R^n \times [0,T)\), we find that
\[u^k(x,h) \geq u_0(x) - C_1\Lambda^{-1}(e^{\Lambda h} - 1) \geq u_0(x) - C_1e^{\Lambda T}h \quad \text{for all } x \in \mathbb R^n.\]
Setting \(v(x,t) = u^k(x,t) - C_1he^{\Lambda(T+t)}\) for \((x,t) \in \mathbb R^n \times [0,T)\), we easily observe that \(v\) is a viscosity subsolution of (\ref{Hk}). We apply comparison principle, to obtain the inequality \(v(x,t) \leq u(x,h+t)\) for all \((x,t) \in \mathbb R^n \times [0,T-h)\), which shows that
\[\liminf_{h \to 0+} \frac{u^k(x,t+h) - u^k(x,t)}{h} \geq -C_1e^{2\Lambda T}.\]
This assures that for any \((x,t) \in  \mathbb R^n \times (0,T)\) and \((p,q) \in D^+u^k(x,t)\), \(q \geq -C_1e^{2\Lambda T}\) and
\[|q| + \hat{H}_k(x,p,u^k(x,t)) \leq q + \hat{H}_k(x,p,u^k(x,t)) + 2C_1e^{2\Lambda T}.\]
Thus, \(u^k\) is a viscosity subsolution of
\begin{equation}
    |u_t| + \hat{H}_k(x,Du,u) - 2C_1e^{2\Lambda T} = 0 \quad \text{in }  \mathbb R^n \times (0,T),
\end{equation}
combining with (\ref{10}), we have that there exists a constant $C_3>0$, such that
\[|u^k_t| + |Du^k| \leq C_3 \quad \text{in } \mathbb R^n \times (0,T) \quad \text{in the viscosity sense}.\]
This shows that \(\{u^k\}_{k \in \mathbb{N}}\) is equi-Lipschitz continuous on \(\mathbb R^n \times [0,T)\). Recalling that \(f^- \leq u^k \leq f^+\) on \(\mathbb R^n \times [0,T)\) for all \(k \in \mathbb{N}\), we find by the Ascoli-Arzela theorem that the family \(\{u^k\}_{k \in \mathbb{N}}\) has a subsequence, converging to some \(u\) in \(C(\mathbb R^n \times [0,T),\mathbb{R})\). Since \(\{\hat{H}_k\}_{k \in \mathbb{N}}\) converges to \(\tilde{H}\) in \(C(\mathbb R^n \times \mathbb R^n\times \mathbb{R},\mathbb{R})\), we find that \(u\) is a viscosity solution of \(u_t + \tilde{H}(x,Du,u) = 0\) in \(\mathbb R^n \times (0,T)\). It is obvious that \(|u| \leq C_2\) on \(\mathbb R^n \times [0,T)\), which implies that \(u\) is a viscosity solution of \(u_t + H(x,Du,u) = 0\) in \(\mathbb R^n \times (0,T)\), \(u\) is Lipschitz continuous on \(\mathbb R^n \times [0,T)\), and \(u(x,0) = u_0(x)\) for all \(x \in \mathbb R^n\). Thus, \(u\) is a Lipschitz continuous solution of (\ref{HJ}). The uniqueness is from the comparison principle.  And since existence is a local property, we can easily extend the result to \(\mathbb R^n \times (0,\infty)\).
\end{proof}

\begin{definition}
    Assume that $H$ satisifes (H2). We say the $H^*$ is the Legendre transform of $H$, if
        \begin{align*}
            H^*(x,p,u)= \underset{v \in \mathbb R^n}{sup} \{v\cdot p- H(x,v,u)\},
        \end{align*}
\end{definition}
where $(x,p,u) \in \mathbb R^n \times \mathbb R^n \times \mathbb R$.

Since $H$ satisfies (H2), $H^*$ is finite, well-defined and also satisfies (H2).

In \cite{KLJ2}, we introduce an implicit variational principle for the Hamilton-Jacobi equation (HJ), which is stated as follows.
\begin{proposition}[Implicit Variational Principle]
    For any $x_0 \in \mathbb R^n$ and $u_0 \in \mathbb R$, there exists a continuous function $h_{x_0,u_0}$ defined on $\mathbb R^n \times (0,\infty)$ satisfying
        \begin{equation}
            h_{x_0,u_0}(x,t)= u_0 + \underset{\gamma (t)=x, \gamma (0)=x_0}{\inf} \int ^t _0 H^*(\gamma (\tau), \overset{\cdot}{\gamma }(\tau), h_{x_0,u_0}(\gamma (\tau),\tau)) d\tau ,\label{picard}
        \end{equation}
    where the infimum is taken among the Lipschitz continuous curves $\gamma:[0,t] \rightarrow \mathbb R^n $ and can be achieved.\label{IVP}
\end{proposition}

With this proposition, we can show the following theorem:

\begin{result}\label{thm1}
    There is a semi-group of the operator $\{S_H(t)\}_{t \geq 0}: C(\mathbb R^n,\mathbb R) \rightarrow C(\mathbb R^n,\mathbb R)$, such that for each $u_0 \in  \text{Lip}(\mathbb R^n)\cap BUC(\mathbb R^n)$, $S_H(t)[u_0(x)]$ is the unique viscosity solution of equation (HJ). Moreover, we have 
        \begin{align}
            S_H(t)[u_0(x)]= \underset{y \in \mathbb R^n}{\inf}\{h_{y,u_0(y)}(x,t)\}, \quad\forall(x,t) \in \mathbb R^n \times [0,\infty)
        \end{align}
    where $h$ is the implicit action function we defined in proposition 1.2. $\{S_H(t)\}_{t \geq 0}$ can also be called Lax-Oleinik semi-group.
\end{result}
Before we prove this theorem, we need to introduce some definitions and lemmas.
\begin{definition}
    Let $T>0$. A function $u:\mathbb R^n \times [0,T] \to \mathbb R$ is called a variational solution of (HJ) if
    \item[(i)] for each continuous and piecewise $C^1$ curve $\gamma : [t_1,t_2]\to \mathbb R^n$ with $0\leq t_1< t_2\leq T$, we have
    \begin{equation}
        u(\gamma(t_2),t_2)-u(\gamma(t_1),t_1)\leq \int _{t_1}^{t_2}H^*(\gamma(\tau),\overset{\cdot}{\gamma}(\tau),u(\gamma(\tau),\tau))d\tau;
    \end{equation}
    \item[(ii)] for each $[t_1,t_2]\subset[0,T]$ and each $x\in \mathbb R^n$, there exists a $C^1$ curve $\gamma: [t_1,t_2]\to \mathbb R^n$ with $\gamma(t_2)=x$ such that
    \begin{equation}
        u(x,t_2)-u(\gamma(t_1),t_1)=\int _{t_1}^{t_2}H^*(\gamma(\tau),\overset{\cdot}{\gamma}(\tau),u(\gamma(\tau),\tau))d\tau.
    \end{equation}
\end{definition}
Refered to \cite{KLJ}, Lemma 4.1, we have:
\begin{lemma}
     Given $u_0 \in C(\mathbb R^n,\mathbb R)$ ,$T>0$ , we define a operator $A:C(\mathbb R^n \times [0,T] ,\mathbb R)\to C(\mathbb R^n \times [0,T] ,\mathbb R)$ by
     \begin{equation}
         A[u](x,t):=\underset{\gamma (t)=x}{\inf}\{u_0(\gamma(0)) +  \int ^t _0 H^*(\gamma (\tau), \overset{\cdot}{\gamma }(\tau),u(\gamma (\tau),\tau)) d\tau\},
     \end{equation}
     where the infimum is taken among the Lipschitz curves $\gamma :[0,t] \to \mathbb R^n$. Then $A$ admits a unique fixed point.\label{fix}
\end{lemma}

Now we denote $u$ as the fixed point in Lemma \ref{fix}, so with \cite{KLJ} propositon 4.3,  the following holds.
\begin{proposition}
    For each $u_0\in  C(\mathbb R^n,\mathbb R)$, we have that
    \begin{equation}
        u(x,t)=\underset{y\in \mathbb R^n}{\inf} h_{y,u_0(y)}(x,t).
    \end{equation}
\end{proposition}
\begin{lemma}
    $u(x,t)$ is a variational solution of (HJ).
\end{lemma}
\begin{proof}
    We need to check that $u$ satisfies (i) and (ii) in Definition 1.3.
    
    (i)Let $\gamma:[t_1,t_2] \to \mathbb R^n$ be a continuous piecewise curve $C^1$ curve. Let $\bar \gamma:[0,t_1]\to \mathbb R^n$be the minimizer of $u(\gamma(t_1),t_1)$. Consider a curve $\xi :[0,t_2]\to \mathbb R^n$ defined by
    \begin{equation}
        \xi(t) = 
        \begin{cases} 
        \tilde{\gamma}(t), & t \in [0, t_1], \\
        \gamma(t), & t \in (t_1, t_2].
        \end{cases}
    \end{equation}
    It follows that
    \begin{align*}
        &u(\gamma(t_2),t_2)-u(\gamma(t_1),t_1)\\
        &=\underset{\gamma_2(t_2)=\gamma(t_2)}{\inf}\{u_0(\gamma_2(0))+\int_0^{t_2}H^*(\gamma_2(\tau),\overset{\cdot}{\gamma_2 }(\tau),u(\gamma_2(\tau),\tau))d\tau\}\\
        &-\underset{\gamma_1(t_1)=\gamma(t_1)}{\inf}\{u_0(\gamma_1(0))+\int_0^{t_1}H^*(\gamma_1(\tau),\overset{\cdot}{\gamma_1 }(\tau),u(\gamma_1(\tau),\tau))d\tau\}\\
        &\leq u_0(\xi(0))+\int_0^{t_2}H^*(\xi(\tau),\overset{\cdot}{\xi }(\tau),u(\xi(\tau),\tau))d\tau-u_0(\bar \gamma(0))-\int_0^{t_2}H^*(\bar \gamma(\tau),\overset{\cdot}{\bar \gamma }(\tau),u(\bar \gamma(\tau),\tau))d\tau,
    \end{align*}
    which implies that 
    \begin{equation}
        u(\gamma(t_2),t_2)-u(\gamma(t_1),t_1)\leq \int _{t_1}^{t_2}H^*(\gamma(\tau),\overset{\cdot}{\gamma}(\tau),u(\gamma(\tau),\tau))d\tau;
    \end{equation}
    (ii) follows from the Markov property for $u$, which is easy to show.
\end{proof}
And with Proposition 7.2.7 in \cite{AF}, we know that $u$ is the viscosity solution of (HJ). So we complete the proof of Theorem 1.

The proof of proposition 1.2 and theorem 1 refer to \cite{KLJ}. The proof for $\mathbb R^n$ is just a extension for the situation on the compact connected set $M$. 

Then, we will provide the main results of this paper.

\begin{result}\label{thm2}
     Assume that Hamiltonian $H$ and $F$ satisfy (H1)-(H5), then semi-groups $S_H$ and $S_F$ commute, i.e. for each $\phi \in \text{Lip}(\mathbb R^n)\cap BUC(\mathbb R^n)$, 
     \begin{align}
         S_H(\lambda) \circ S_F(\mu)[\phi(x)]= S_F(\mu) \circ S_H(\lambda)[\phi(x)], \quad\forall x \in \mathbb R^n, \lambda \geq 0, \mu \geq 0.
     \end{align}
\end{result}

Before we prove this theorem, we need to introduce some other theorems first.

\begin{result}\label{thm3}
     Assume that $H$ and $F$ satisfy (H1)(H2)(H4)(H5). $v(x,t,\lambda)$ is the unique Lipschitz continuous viscosity solution of (HJ) with Hamiltonian $G=H+\lambda F$ where $\lambda \geq 0$. Then
     \item[(i)] If $F$ satisfies (H3), for every $t \geq 0$, $v$ is a viscosity subsolution of 
        \begin{equation}
            \left\{
            \begin{aligned}
            & \frac{\partial \xi}{\partial \lambda} +tF(x,D\xi,\xi)=0 \quad in \quad \mathbb R^n \times (0,\infty),
            \\
            &\xi(x,0)=u_0(x) \quad in \quad \mathbb R^n.
            \end{aligned}
            \right.\label{1}
        \end{equation}
    \item[(ii)] If $F$ and $H$ satisfy (H3), then for each $t \geq 0$, $v$ is a viscosity solution of (\ref{1}).  
\end{result}
The proof of theorem 3 is left in the next section. Then we consider the multi-time Hamilton system, which is stated as follows:
\begin{equation}\tag{HJ1}
    \frac{\partial u}{\partial t_1} +H_1(x,Du,u)=0 \quad in \quad \mathbb R^n \times [0,\infty)^d,
\end{equation}
\begin{equation}\tag{HJ2}
    \frac{\partial u}{\partial t_2} +H_2(x,Du,u)=0 \quad in \quad \mathbb R^n \times [0,\infty)^d,
\end{equation}
\begin{equation*}
    ......
\end{equation*}

\begin{equation}\tag{HJd}
    \frac{\partial u}{\partial t_d} +H_d(x,Du,u)=0 \quad in \quad \mathbb R^n \times [0,\infty)^d,
\end{equation}
\begin{equation}\tag{HJ0}
    u(x,0,0,...,0)=u_0(x),\quad u_0\in \text{Lip}(\mathbb R^n)\cap BUC(\mathbb R^n).
\end{equation}

\begin{result}
    Assume that $H_1$, $H_2$,..., $H_d$ satisfy (H1)-(H4) and any pair $H_i$, $H_j$ $(i \neq j)$ satisfy (H5). Then there exists a unique viscosity solution of (HJ1)-(HJ2)-...-(HJd)-(HJ0) in $W^{1,\infty}(\mathbb R^n \times [0,T]^d)$ for each $T>0$.
\end{result}

\begin{proof}
    Let $U(x,s,t_1,...,t_d)$ be the unique locally Lipschitz continuous viscosity solution of
    \begin{equation*}
        \frac{\partial U}{\partial s}+t_1H_1(x,DU,U)+...+t_dH_d(x,DU,U)=0 \quad in \quad \mathbb R^n \times (0,\infty)
    \end{equation*}
    \begin{equation*}
        U(x,0,t_1,...,t_d)=u_0(x) \quad in \quad \mathbb R^n
    \end{equation*}
    Then setting $u(x,t_1,...,t_d)=U(x,1,t_1,...,t_d)$.
    Applying Theorem 3 (ii) with $F=H_1$, $\lambda=t_1$ and $H=t_2H_2+...+t_dH_d$, for $s=1$, $u$ is a viscosity solution of  
    \begin{equation}
    \left\{
   \begin{aligned}
   & \frac{\partial u}{\partial t_1}(x,t)+H_1(x,Du(x,t),u(x,t))=0\quad \textrm{in}\quad \mathbb R^n\times(0,\infty)
   \\
   &u(x,0)=u_0(x),\quad in \quad \mathbb R^n,
   \\
   \end{aligned}
   \right.
\end{equation}
where $t:=(t_1,t_2,...,t_d)$. Similar arguments for $H_2$,...,$H_d$. We know that $u$ satisfies (H1)-(HJd). And when $t_1=t_2=...=t_d=0$, the equation changes to $\frac{\partial u}{\partial s}=0$, so $u(x,0,...,0)=U(x,1,0,...,0)=U(x,0,...,0)=u_0(x)$. Therefore $u$ also holds for the initial condition.
\end{proof}

More specifically, we have the following remark.

\begin{remark}
    Assume that $H$ satisy (H1)-(H4), then for each $t >0$, $u_0 \in W^{1,\infty}(\mathbb R^n)$, we have that
    \begin{equation}
        S_H(t)[u_0(x)]=S_{tH}(1)[u_0(x)] \quad in \quad \mathbb R^n.
    \end{equation}
\end{remark}
This is yielded by considering when $H=0$ in Theorem 3, then we know that $S_{\mu H}(1)=S_H(\mu)$ with the uniqueness of viscosity solution.

We will use Theorem 3 and Theorem 4 to prove Theorem 2. 
\begin{proof}[Proof of Theorem 2]
    Consider $s,\lambda,\mu \geq 0$.
    With Theorem 3 (ii) and Theorem 4, let $v(x,s,\lambda,\mu)=S_{\mu H+\lambda F}(s)[u_0(x)]$, we have
    \begin{align*}
        & S_{\mu H+\lambda F}(1)[u_0(x)]= S_F(\lambda)[v(x,1,0,\mu)]\\
        &= S_F(\lambda) \circ S_{\mu H} (1)[u_0(x)]= S_F(\lambda) \circ S_H(\mu) [u_0(x)]
    \end{align*}
\end{proof}

\section{Proof of Theorem 3}
We first introduce a lemma.
\begin{lemma}
    $F$, $H$ satisfy (H1)-(H5), denote
    \begin{equation}
        v^{\epsilon}(x,t,\lambda):=\underset{y\in \mathbb R^n}{\inf} \{ v(y,t,\lambda)+t\epsilon F^*\big(x,\frac{x-y}{t\epsilon},v(y,t,\lambda)\big)\},
    \end{equation}
    where $v(y,t,\lambda)$ is defined in Theorem 3.
    \item[(i)] $v^{\epsilon} \rightarrow v$ locally uniformly in $\mathbb R^n$ as $\epsilon \rightarrow 0$.
    \item[(ii)] For each $R$, $T$, $\Lambda>0$, $\epsilon$ small enough, $v^\epsilon$ is a viscosity supersolution of 
    \begin{equation}
        \frac{\partial w}{\partial t} +H(x,Dw,w)+(\lambda +\epsilon)F(x,Dw,w)=-o_{R,T,\Lambda}(\epsilon) \quad in \quad B(0,R) \times (0,T),
    \end{equation}
    for any $0 \leq \lambda \leq \Lambda$, where $o_{R,T,\Lambda}$ is a function such that $o_{R,T,\Lambda}(\epsilon) /\epsilon\rightarrow 0$ as $\epsilon \rightarrow 0$ for any fixed $R,T,\Lambda$.
     \item[(iii)] For each $R$, $T$, $\Lambda>0$, $\epsilon$ small enough, $v^\epsilon$ is a viscosity subsolution of 
    \begin{equation}
        \frac{\partial w}{\partial t} +H(x,Dw,w)+(\lambda +\epsilon)F(x,Dw,w)=o_{R,T,\Lambda}(\epsilon) \quad in \quad B(0,R) \times (0,T),\label{3}
    \end{equation}
    for any $0 \leq \lambda \leq \Lambda$, where $o_{R,T,\Lambda}$ is a function such that $o_{R,T,\Lambda}(\epsilon) /\epsilon\rightarrow 0$ as $\epsilon \rightarrow 0$ for any fixed $R,T,\Lambda$.
\end{lemma}
\begin{proof}
    We first prove (i). Since $v(y,t,\lambda)+t\epsilon F^*\big(x,\frac{x-y}{t\epsilon},v(y,t,\lambda)\big) \rightarrow +\infty$ as $|y| \rightarrow \infty$, we know that there exists $y_{\epsilon} \in \mathbb R^n$, s.t. $v^{\epsilon}(x,t,\lambda)= v(y_{\epsilon},t,\lambda)+t\epsilon F^*\big(x,\frac{x-y_{\epsilon}}{t\epsilon},v(y_{\epsilon},t,\lambda)\big)$. 
    
Let $k:=||Dv(x,t,\lambda)||_{L^{\infty}(\mathbb R^n)}(1+||\frac{\partial F^*}{\partial u}||_{L^{\infty}})$, with (H2) we know that there exists $C$, s.t.
    \begin{equation}
        \frac{F^*(x,p,u)-F^*(x,0,u)}{|p|}>k, \quad \forall |p|>C, x \in \mathbb R^n, u\in \mathbb R.\label{2}
    \end{equation}
    Therefore 
    \begin{align*}
        &v(y_{\epsilon},t,\lambda)+t\epsilon F^*\big(x,\frac{x-y_{\epsilon}}{t\epsilon},v(y_{\epsilon},t,\lambda)\big)\\
        & \leq v(x,t,\lambda)+t\epsilon F^*(x,0,v(x,t,\lambda))\\
        & \leq  v(y_{\epsilon},t,\lambda)+t\epsilon F^*\big(x,0,v(y_{\epsilon},t,\lambda)\big)+k|x-y_{\epsilon}|.
    \end{align*}
    With (\ref{2}), we know that $|x-y_{\epsilon}| \leq t\epsilon C$. Since $|x-y_{\epsilon}|/\epsilon t \leq C$, then we have
    \begin{align*}
        &\epsilon tkR \geq \epsilon t F^*(x,0,v(x,t,\lambda))\\
        &\geq v^{\epsilon}(x,t,\lambda)-v(x,t,\lambda)\\
        &\geq -k|x-y_{\epsilon}|-K_C \epsilon t\\
        &\geq -C^{\prime}\epsilon t,
    \end{align*}
    where $K_C$ is defined on (H1), $C^{\prime}$ is a positive constant independent of $\epsilon$ and $t$. Thus $v^{\epsilon} \rightarrow v$ locally uniformly in $\mathbb R^n \times [0,\infty)$ as $\epsilon \rightarrow 0$.
    
    We only prove (iii) since the argument of (ii) follows from more classical methods and similar computation. From \cite{HT} Theorem 2.27, we know that we only need to check that $v^{\epsilon}$ is an almost everywhere subsolution of (\ref{3}). Since $v^{\epsilon}$ is Lipschitz in $B(0,R) \times (0,T)$, $v^{\epsilon}$ is differentiable a.e. Therefore it is sufficient to prove that the inequality holds at each point of differentiability of $v^{\epsilon}$. Let $(x,t) \in B(0,R) \times (0,T)$ be a point of differentiability of $v^{\epsilon}$. Set 
    \begin{equation}
        \Psi _{\epsilon}(x,y,t)=\epsilon tF^*\big(x,\frac{x-y}{t\epsilon},v(y,t,\lambda)\big),
    \end{equation}
    then there exists $\bar y \in \mathbb R^n$, s.t.
    \begin{equation}
        v^{\epsilon}(x,t,\lambda)=v(\bar y,t,\lambda)+\Psi_{\epsilon}(x,\bar y,t).
    \end{equation}
    Because $(Dv^{\epsilon}(x,t,\lambda),\frac{\partial v^{\epsilon}}{\partial t}(x,t,\lambda)) \in D^-v^{\epsilon}(x,t,\lambda)$, from \cite{HT} Theorem 1.4, we know that there exists a $C^1$ function $\Phi$ s.t. $v^{\epsilon}- \Phi$ attains its global minimum at $(x,t)$, and 
    \begin{equation}
        \frac{\partial \Phi}{\partial t}(x,t)=\frac{\partial v^{\epsilon}}{\partial t}(x,t,\lambda), \quad D\Phi(x,t)=Dv^{\epsilon}(x,t,\lambda).
    \end{equation}
    So there exists $(p,a)\in D^-v(\bar y,t,\lambda)$ s.t.
    \begin{align}
        & a+\frac{\partial \Psi _{\epsilon}}{\partial t}(x,\bar y,t)=\frac{\partial v^{\epsilon}}{\partial t}(x,t,\lambda) \\
        &p=-D_y \Psi _{\epsilon}(x,\bar y,t), \quad D_x \Psi _{\epsilon}(x,\bar y,t)=D_x v^{\epsilon}(x,t,\lambda).
    \end{align}
    With \cite{HT} Theorem 2.29, we have
    \begin{equation}
        a+H(\bar y,p,v(\bar y,t,\lambda))+\lambda F(\bar y,p,v(\bar y,t,\lambda))=0
    \end{equation}
    Now we estimate
    \begin{equation}
        Q:=\frac{\partial v^{\epsilon}}{\partial t} +H(x,Dv^{\epsilon},v^{\epsilon})+(\lambda +\epsilon)F(x,Dv^{\epsilon},v^{\epsilon})
    \end{equation}
    at the point $(x,t,\lambda)$. Note that $q:=\frac{x-\bar y}{\epsilon t}$, we have
    \begin{align}
        & \frac{\partial v^{\epsilon}}{\partial t}(x,t,\lambda)=a+\epsilon[F^*(x,q,v(\bar y,t,\lambda)-D_pF^*(x,q,v(\bar y,t,\lambda)) \cdot q+ta \frac{\partial F^*}{\partial u}(x,q,v(\bar y,t,\lambda))]\label{4}\\
        &Dv^{\epsilon}(x,t,\lambda)=t\epsilon D_xF^*(x,q,v(\bar y,t,\lambda))+D_pF^*(x,q,v(\bar y,t,\lambda))\\
        &p=D_pF^*(x,q,v(\bar y,t,\lambda))-\epsilon t\frac{\partial F^*}{\partial u}(x,q,v(\bar y,t,\lambda))\cdot p
    \end{align}
    To simplify the computation, we denote that 
    \begin{equation}
        v_1:=v(x,t,\lambda),v_2:=v(\bar y,t,\lambda),v^{\epsilon}_1:=v^{\epsilon}(x,t,\lambda).\label{5}
    \end{equation}
    And if $\bar r \in \mathbb R^n$ s.t.
    \begin{equation}
        F^*(x,q,v_2)=\bar r \cdot q-F(x,\bar r,v_2)=\underset{r\in \mathbb R^n}{\sup} \{r\cdot q-F(x,r,v_2)\}
    \end{equation}
    Then $D_xF^*(x,q,v_2)=-D_xF(x,\bar r,v_2)$,$\frac{\partial F}{\partial u}(x,\bar r,v_2)=-\frac{\partial F^*}{\partial u}(x,q,v_2)$ and $D_pF^*(x,q,v_2)=\bar r$. Therefore $p(1-\epsilon t\frac{\partial F}{\partial u}(x,\bar r,v_2))=\bar r$. Combine (\ref{4})-(\ref{5}), we have
    \begin{align*}
        &\frac{\partial v^{\epsilon}}{\partial t}(x,t,\lambda)=a+\epsilon[F^*(x,q,v_2)-\bar r \cdot q+ta \frac{\partial F^*}{\partial u}(x,q,v_2)]\\
        &=-\epsilon F(x,\bar r,v_2)+a(1-\epsilon t\frac{\partial F}{\partial u}(x,\bar r,v_2)).
    \end{align*}
    Rewrite that 
    \begin{equation}
        Q=-\epsilon F(x,\bar r,v_2)+a(1-\epsilon t\frac{\partial F}{\partial u}(x,\bar r,v_2))+H(x,\bar r+p_2,v^{\epsilon}_1)+(\lambda+\epsilon)F(x,\bar r+p_2,v^{\epsilon}_1),
    \end{equation}
    where we denote that $p_2:=\epsilon tD_xF^*(x,q,v)$. Furthermore, consider (2.10), we have 
    \begin{equation*}
        Q=\epsilon[F(x,\bar r+p_2,v^{\epsilon}_1)-F(x,p,v_2)]+[H(x,\bar r+p_2,v^{\epsilon}_1)-H(\bar y,p,v_2)-\epsilon tH(\bar y,p,v_2)\frac{\partial F}{\partial u}(x,\bar r,v_2)]
    \end{equation*}
    \begin{equation}
        +\lambda [F(x,\bar r+p_2,v^{\epsilon})-F(\bar y,p,v)-\epsilon tF(\bar y,p,v_2)\frac{\partial F}{\partial u}(x,\bar r,v_2)]
    \end{equation}
    Since $v$ is Lipschitz continuous, $p$ and $\bar r$ are bounded, which implies that 
    \begin{equation}
        |p_2|=tO_{R,T,\Lambda}(\epsilon),
    \end{equation}
    where $O_{R,T,\Lambda}$ is a function s.t. $\frac{O_{R,T,\Lambda}(\epsilon)}{\epsilon}$ is bounded.
   
   It is obvious that 
    \begin{equation}
        \epsilon[F(x,p+p_2,v^{\epsilon})-F(x,p,v)]=o_{R,T,\Lambda}(\epsilon).
    \end{equation}
    Besides, we notice that
    \begin{equation}
        F(x,\bar r,v_2)=\bar r\cdot q-F^*(x,q,v_2)=\underset{r\in \mathbb R^n}{\sup} \{r\cdot \bar r-F^*(x,r,v_2)\},
    \end{equation}
    therefore
    \begin{equation}
        D_xF(x,\bar r,v_2)=-D_xF^*(x,q,v_2),\quad D_pF(x,\bar r,v_2)=q,
    \end{equation}
    which implies that\
    \begin{equation}
        p_2=\epsilon tD_xF^*(x,q,v_2)=-\epsilon t D_xF(x,\bar r,v_2), \quad \bar y=x-\epsilon tD_pF(x,\bar r,v_2).
    \end{equation}
    We define an operator 
    \begin{equation}
        \bar Q(G)=G(x,\bar r-\epsilon tD_xF(x,p,v_2),v^{\epsilon}_1)-G(x-\epsilon tD_pF(x,p,v_2),p,v_2)-\epsilon tG(\bar y,p,v_2)\frac{\partial F}{\partial u}(x,\bar r,v_2),
    \end{equation}
    where $G=H$ or $F$.
    We rewrite $\bar Q$ as following:
 \begin{align*}
&\bar Q(G) = G\bigl(x,\bar r-\epsilon tD_xF(x,\bar r,v_2),v_2+\epsilon t(\bar r\cdot q-F(x,\bar r,v_2))\bigr) \\
&\quad -G\bigl(x-\epsilon tD_pF(x,\bar r,v_2),p,v_2\bigr)-\epsilon tG(\bar y,p,v_2)\frac{\partial F}{\partial u}(x,\bar r,v_2)\\[4pt]
&= \int_0^1 \frac{d}{d\tau} G\Bigl(x-\epsilon t(1-\tau)D_pF(x,\bar r,v_2), \\
&\qquad \bar r-\epsilon t\tau D_xF(x,\bar r,v_2)+\epsilon t(1-\tau)\frac{\partial F}{\partial u}(x,\bar r,v_2)p, \\
&\qquad v_2+\epsilon t\tau (\bar r\cdot q-F(x,\bar r,v_2))\Bigr) d\tau \\
&\quad -\epsilon tG(\bar y,p,v_2)\frac{\partial F}{\partial u}(x,\bar r,v_2)\\[4pt]
&= \epsilon t\Bigl[ D_x G(x,p,v_2) D_pF(x,p,v_2)- D_p G(x,p,v_2) D_xF(x,p,v_2) \\
&\qquad + p \cdot D_pF(x,p,v_2) \cdot \frac{\partial G}{\partial u}(x,p,v_2) \\
&\qquad - p \cdot D_pG(x,p,v_2) \cdot \frac{\partial F}{\partial u}(x,p,v_2) \\
&\qquad + \frac{\partial F}{\partial u}(x,p,v_2) G(x,p,v_2)-\frac{\partial G}{\partial u}(x,p,v_2) F(x,p,v_2)\Bigr] \\
&\quad + o_{R,T,\Lambda}(\epsilon),
\end{align*}
   here we used that $|p-\bar r|=tO_{R,T,\Lambda}(\epsilon)$. Then with (H5), we deduce that $\bar Q(G)=o_{R,T,\Lambda}(\epsilon)$, for $G=H$ or $F$. Therefore
\begin{align*}
    & Q=\epsilon[F(x,p+p_2,v^{\epsilon})-F(x,p,v)]+\bar Q(H)+\lambda \bar Q(F)\\
    &=o_{R,T,\Lambda}(\epsilon),
\end{align*}
   
\end{proof}
Now we come back to Theorem 3.

\begin{proof}[Proof of Theorem 3]
With Lemma 2.1, we know that %%%%%换定理引用
\begin{equation}
    w_1(x,t):=v^{\epsilon}(x,t,\lambda)+to_{R,T,\Lambda}(\epsilon)
\end{equation}
is a supersolution of the equation:
\begin{equation}
    \frac{\partial w}{\partial t} +H(x,Dw,w)+(\lambda +\epsilon)F(x,Dw,w)=0 \quad in \quad B(0,R) \times (0,T)
\end{equation}
With the comparison principle for contact type Hamilton-Jacobi equations, which refers to \cite{ishii2022hamilton} Lemma A.2, for every $0 \leq \lambda \leq \Lambda$, $R^\prime >0$ and $0 \leq t \leq T$, there exists $R>0$ large enough such that, for $\epsilon$ small enough, we have
\begin{equation}
    v(x,t,\lambda+\epsilon) \leq v^{\epsilon}(x,t,\lambda)+to_{R,T,\Lambda}(\epsilon) \quad in \quad B(0,R^\prime) \times \{t\}.
\end{equation}

Furthermore,
\begin{equation}
    v(x,t,\lambda+\epsilon) \leq v(y,t,\lambda)+\epsilon tF^*\big(x,\frac{x-y}{t\epsilon},v(y,t,\lambda)\big)+to_{R,T,\Lambda}(\epsilon) \quad in \quad B(0,R^\prime) \times \{t\}
\end{equation}
for every $y \in \mathbb R^n$. Now if $v$ is $C^1$, for any $q \in \mathbb R^n$, we pick $y \in \mathbb R^n$ s.t.$\frac{x-y}{\epsilon t}=q$, then we learn that 
\begin{equation}
    0\leq -Dv(x,t,\lambda)\cdot \epsilon tq-\epsilon \frac{\partial v}{\partial \lambda}(x,t,\lambda)+to_{R,T,\Lambda}(\epsilon)+o(\epsilon)\quad in \quad B(0,R^\prime) \times \{t\},
\end{equation}
Dividing $\epsilon$ and letting $\epsilon \rightarrow 0$, 
\begin{equation}
    \frac{\partial v}{\partial \lambda}(x,t,\lambda)+t[Dv(x,t,\lambda)\cdot q-F^*(x,q,v(x,t,\lambda))]\leq 0
\end{equation}
Finally,for any $t>0$,$R^\prime >0$,$\Lambda >0$, we take the supremum for $q\in \mathbb R^n$,
\begin{equation}
    \frac{\partial v}{\partial \lambda}(x,t,\lambda)+tF(x,Dv(x,t,\lambda),v(x,t,\lambda))\leq 0, \quad in \quad (x,\lambda)\in B(0,R^\prime) \times (0,\Lambda).
\end{equation}
Similarily, we can prove that $v$ is the supersolution of (\ref{1}).
\end{proof}
\section{Further results and extensions}
\begin{remark}
    There is a question left : where is the condition (H5) from? 
    
    We use an approximation:
    \begin{equation}
        u(x,\Delta t)=u_0(x)-H(x,Du_0(x),u_0(x))\Delta t
    \end{equation}
    where $u$ is the viscosity solution of (HJ), and $\Delta t$ is small. To simplify the computation, we denote that $p:=Du_0(x)$, $u:=u_0(x)$.
    Thus 
    \begin{align*}
        &S_H \circ S_F(\Delta t)[u_0(x)]=S_H(\Delta t)[u_0(x)-F(x,Du_0(x),u_0(x))\Delta t]\\
        &=u-F(x,p,u)\Delta t-H(x,p-D_x(F(x,Du_o(x),u_0(x))),u-F(x,Du_0(x),u_0(x))\Delta t)\Delta t\\
        &=u-F(x,p,u)\Delta t-H(x,p-D_xF(x,p,u)\Delta t-D^2u_0(x)D_pF(x,p,u)\Delta t-\frac{\partial F}{\partial u}p\Delta t,\\
        &u-F(x,Du_0(x),u_0(x))\Delta t)\Delta t\\
        &=u-F(x,p,u)\Delta t-H(x,p,u)\Delta t-D_pH(x,p,u)\cdot D_xF(x,p,u)(\Delta t)^2-\\
        &D_pH^TD^2u_o(x)D_pF(\Delta t)^2-\frac{\partial F}{\partial u}p\cdot D_pH(\Delta t)^2-\frac{\partial H}{\partial u}F(\Delta t)^2,
    \end{align*}
    With $S_H \circ S_F(\Delta t)[u_0(x)]=S_F \circ S_H(\Delta t)[u_0(x)]$ for all $u_0\in C^1(\mathbb R^n)$, we get (H5).
\end{remark}
With this remark, we can know inequality. We modify the condition (H5).
\begin{itemize}
    \item[(H5.1)] $H$ and $F$ satisfies 
     \begin{align*}
          D_x H(x,p,u) D_pF(x,p,u)- D_p H(x,p,u) D_xF(x,p,u) + p \cdot D_pF(x,p,u) \cdot \frac{\partial H}{\partial u} 
          \\ - p \cdot D_pH(x,p,u) \cdot \frac{\partial F}{\partial u} + \frac{\partial F}{\partial u}(x,p,u) H(x,p,u)-\frac{\partial H}{\partial u}(x,p,u) F(x,p,u)\leq 0,
     \end{align*}
     for every $(x,p,u) \in \mathbb R^n \times \mathbb R^n \times \mathbb R$
\end{itemize}
\begin{result}
     Assume that Hamiltonian $H$ and $F$ satisfy (H1)-(H4) and (H5.1), then for each $\phi \in  C(\mathbb R^n,\mathbb R)$, 
     \begin{align}
         S_H(\lambda) \circ S_F(\mu)[\phi(x)]\leq S_F(\mu) \circ S_H(\lambda)[\phi(x)], \quad\forall x \in \mathbb R^n, \lambda \geq 0, \mu \geq 0.
     \end{align}
\end{result}
The proof is the same as the arguement for Theorem 2.

Now we will consider the question of whether the condition (H1)-(H5) can be relaxed.
\begin{proposition}
    Assume that $\{H^{\epsilon}\}_{\epsilon >0}$ and $\{F^{\epsilon}\}_{\epsilon >0}$ are two sequences of $C^1$ Hamiltonian, satisfying (H1)-(H4), and $(H^{\epsilon},F^{\epsilon})$ satisfies (H5) for any $\epsilon >0$. If $H^{\epsilon}$ and $F^{\epsilon}$ converges to  $H$ and $F$ respectively, then $S_H$ and $S_F$ commute.
\end{proposition}
With \cite{HT} Theorem 1.13, we know that $S_H$ and $S_F$ are well-defined and commute. Indeed, in this proposition, we do not need that $H$ and $F$ are $C^1$. In this section, all the Hamiltonians are more general, are locally Lipschitz.

\begin{result}
    Assume that $H$ is locally Lipschitz satisfying (H1)-(H4), and $f:\mathbb R \mapsto \mathbb R$ is an increasing, convex continuous function. Then the semi-groups $S_H$ and $S_{f(H)}$ commute.
\end{result}
\begin{proof}
    Let ${\eta}^{\epsilon}$ be the standard mollifier in $\mathbb R$, denote that $f^{\epsilon}(x):=f*\eta^{\epsilon}(x)$. Then we know that $f^{\epsilon}\in C^{\infty}$, $f^{\epsilon} \rightarrow f$ locally uniformly as $\epsilon \rightarrow 0$. Similarly we can define $\{H^{\epsilon}\}_{\epsilon >0}$, 
 s.t. $H^{\epsilon} \rightarrow H$ locally uniformly. Then we can check that $f^{\epsilon}(H^{\epsilon})$, $H^{\epsilon}$ satisfy (H1)-(H4) and $(f^{\epsilon}(H^{\epsilon}),H^{\epsilon})$ satisfies (H5). With proposition 3.1 we know that $S_H$ and $S_{f(H)}$ commute.
\end{proof}
\begin{remark}
    With Theorem 6, we can deduce that $S_H$ and $S_{H^+}$ commute, where $H^+=max\{H,0\}$. Furthermore, assume that $H$ and $F$ are $C^1$ and satisfy (H1)-(H5). If $f$ and $: \mathbb R \rightarrow \mathbb R$ are increasing, convex continuous functions, then $S_{f(H)}$ and $S_{g(F)}$ commute.
\end{remark}
Now we will consider (H5) in almost everywhere sense.
\begin{result}
    Assume that $H$ and $F$ satisfy (H1)-(H4), and $(H,F)$ satisfies (H5) in the almost everywhere sense. If there exist sequences $\{\phi_n\}_{n\in \mathbb N}$ and $\{\psi_n\}_{n\in \mathbb N}$ of $C^1$ functions from $\mathbb R$ to $\mathbb R$, s.t. $\phi_n(H)$ and $\psi_n(F)$ are $C^1$ and convex with respect to $p$, and converge to $H$ and $F$ respectively locally uniformly, then $S_H$ and $S_F$ commute.
\end{result}
The proof is similar to Theorem 6. We will give an example of this type of Hamiltonian.
\begin{example}
    Consider Hamiltonian $H(x,p,u)=a(x)|p|+\sin u$, where $a:\mathbb R^n \rightarrow \mathbb R$ is a $C^1$, Lipschitz function such that
    \begin{equation*}
        0<\alpha \leq a(x)\leq \beta \quad in \quad \mathbb R^n.
    \end{equation*}
    Pick $\phi_n(t)=t^{1+1/n}$, then $\phi_n(H)$ is $C^1$ and convex w.r.t $p$, and $\phi_n(H)$ converges to $H$ locally uniformly.
\end{example}
Moreover, we have the following theorem:
\begin{result}
    Assume that $H$ and $F$ locally Lipschitz satisfy (H1)-(H4). If $F$ is $C^1$ and if $H$ and $F$ satisfy (H5) in the almost everywhere sense, then $S_H$ and $S_F$ commute.
\end{result}
The proof is left in the next section.

In the next part, we will consider the problem in the general manifold. More details about this part can refer to \cite{SF}.

Assume that $M$ is a closed smooth manifold, Hamiltonian $H$ is a real-valued function defined on the jet bundle $\mathbb J^1(M,\mathbb R)$. 
\begin{definition}
    A Jacobi structure on a manifold $M$ is a bracket $\{\cdot,\cdot\}$ on $C^{\infty}(M)$ that is skew-symmetric, satisfies the Jacobi identity, and is local, in the sense that the support of $\{f, g\}$ is contained in the intersection of the supports of $f$ and $g$.
\end{definition}
We define the Jacobi structure for $\mathbb J^1(M,\mathbb R)$, with the following map:
\begin{equation}
    \Lambda ^\#: T^*M \mapsto TM, \quad \eta \mapsto \eta_p\partial_x-(p\eta_u+\eta_x)\partial_p-p\eta_p\partial_u
\end{equation}
and the Reeb field $Y_\alpha =\partial_u$. Define that $\Lambda(\eta,\xi)=\xi(\Lambda ^\#(\eta))$, we give the Jacobi bracket as stated:
\begin{equation}
    \{f,g\}:=\Lambda(df,dg)+(fdg-gdf)(Y_\alpha).
\end{equation}
We call $(C^{\infty}(M)),\{\cdot,\cdot\})$ the Jacobi algebra, and we say that $f$ and $g$ Jacobi commute if $\{f,g\}=0$ for every $(x,p,u)\in T^*M\times \mathbb R$. We modify the $|\cdot|:\mathbb R^n \rightarrow [0,\infty)$ in (H1)-(H3) to the norm $||\cdot||$ on the manifold $M$ at this part. Then we can define the Lax-Oleinik semi-group $S_H$ similarly.  
\begin{result}
    Assume that $H$ and $F$ satisfy (H1)-(H4), if $H$ and $F$ Jacobi commute, then $S_H$ and $S_F$ commute.
\end{result}
We can find that $H$ and $F$ Jacobi commute just means (H5) when $M=\mathbb R^n$. About the proof, we can modify the proof of Theorem 3 by using local coordinates. 

\section{Proof of Theorem 8}
We first modify the proof of Lemma 2.1, we claim that it still holds when (H5) is in the almost everywhere sense.
\begin{equation}
        \bar Q(H)=H(x,\bar r-\epsilon tD_xF(x,\bar r,v_2),v^{\epsilon}_1)-H(x-\epsilon tD_pF(x,\bar r,v_2),p,v_2)-\epsilon tH(\bar y,p,v_2)\frac{\partial F}{\partial u}(x,\bar r,v_2),
 \end{equation}
 Consider the mollifier $\eta \in C^\infty_c (\mathbb R^n \times \mathbb R^n \times \mathbb R)$, such that $\int_{\mathbb R^{2n+1}} \eta dm=1, \eta \geq 0,supp(\eta) \subset B(0,1)$. Let
 \begin{equation}
     \eta^\delta(x,p,u):=\delta^{-2n-1}\eta(\frac{x}{\delta},\frac{p}{\delta},\frac{u}{\delta}),
 \end{equation}
 and 
 \begin{equation}
     H^\delta(x,p,u):=(H*\eta^\delta)(x,p,u).
 \end{equation}
With the property of mollifier, we know that for any compact set $K$, there exists a constant $C>0$, such that
\begin{equation}
    ||H^\delta-H||_{L^\infty(K)}\leq C\delta,
\end{equation}
so we can rewrite that
\begin{equation}
    \bar Q(H)=H^\delta(x,\bar r-\epsilon tD_xF(x,p,v_2),v^{\epsilon}_1)-H^\delta(x-\epsilon tD_pF(x,p,v_2),p,v_2)-\epsilon tH^\delta(\bar y,p,v_2)\frac{\partial F}{\partial u}(x,\bar r,v_2)+O(\delta)
\end{equation}
Moreover, like Lemma 2.1, we have
\begin{align*}
    &H^\delta(x,\bar r-\epsilon tD_xF(x,p,v_2),v^{\epsilon}_1)-H^\delta(x-\epsilon tD_pF(x,p,v_2),p,v_2)-\epsilon tH^\delta(\bar y,p,v_2)\frac{\partial F}{\partial u}(x,\bar r,v_2)\\
    &=\epsilon t[D_x H^\delta(x,p,v_2) D_pF(x,p,v_2)- D_p H^\delta(x,p,v_2) D_xF(x,p,v_2) + p \cdot D_pF(x,p,v_2) \cdot \frac{\partial H^\delta}{\partial u}(x,p,v_2) \\
    & - p \cdot D_pH^\delta(x,p,v_2) \cdot \frac{\partial F}{\partial u}(x,p,v_2) + \frac{\partial F}{\partial u}(x,p,v_2) H^\delta(x,p,v_2)-\frac{\partial H^\delta}{\partial u}(x,p,v_2) F(x,p,v_2)]+o_{R,T,\Lambda}(\epsilon)\\
    &=\int _{\mathbb R^{2n+1}}\epsilon t[D_x H(y,q,v) D_pF(x,p,v_2)- D_p H(y,q,v) D_xF(x,p,v_2) + p \cdot D_pF(x,p,v_2) \cdot \frac{\partial H}{\partial u}(y,q,v) \\
    & - p \cdot D_pH(y,q,v) \cdot \frac{\partial F}{\partial u}(x,p,v_2) + \frac{\partial F}{\partial u}(x,p,v_2) H(y,q,v)-\frac{\partial H}{\partial u}(y,q,v) F(x,p,v_2)]\\
    &\eta ^\delta(x-y,p-q,v_2-v)dydqdv+o_{R,T,\Lambda}(\epsilon)\\
    &=\epsilon t \omega(\delta)+o_{R,T,\Lambda}(\epsilon),
\end{align*}
where $\omega$ is a positive function satisfying $\omega(\epsilon) \rightarrow 0$ when $\epsilon \rightarrow 0$.
Thus we have 
\begin{equation}
    \bar Q(H)=\epsilon t\omega(\delta)+O(\delta)+o_{R,T,\Lambda}(\epsilon),
\end{equation}
pick $\delta=\epsilon^2$, we know that $\bar Q(H)=o_{R,T,\Lambda}^\prime(\epsilon)$.
So $v$ is the viscosity solution of
\begin{equation}
    \frac{\partial v}{\partial \lambda}+tF(x,Dv,v)=0 \quad in \quad\mathbb R^n \times (0,\infty),\label{7}
\end{equation}
but the same argument does not hold for
\begin{equation}
    \frac{\partial v}{\partial \lambda}+tH(x,Dv,v)=0 \quad in \quad\mathbb R^n \times (0,\infty).\label{8}
\end{equation}
We will give another lemma.
\begin{lemma}
    For any $k>0$, we have 
    \begin{equation}
        v(x,t,\lambda,\mu)=v(x,\frac{t}{k},k\lambda,k\mu),
    \end{equation}
    and $v$ is the viscosity solution of 
    \begin{equation}
        t\frac{\partial v}{\partial t}-\lambda\frac{\partial v}{\partial \lambda}-\mu\frac{\partial v}{\partial \mu}=0 \quad in \quad \mathbb R^n\times(0,\infty)^3, \label{6}
    \end{equation}
\end{lemma}
\begin{proof}
    It is easy to check that $v(x,\frac{t}{k},k\lambda,k\mu)$ is also the viscosity solution of 
    \begin{equation}
        \frac{\partial v}{\partial t}+\mu H(x,Dv,v)+\lambda F(x,Dv,v)=0 \quad in \quad \mathbb R^n \times (0,\infty),\label{11}
    \end{equation}
    due to the uniqueness of viscosity solution we deduce that $v(x,t,\lambda,\mu)=v(x,\frac{t}{k},k\lambda,k\mu)$.
    Suppose that $\phi \in C^1(\mathbb R^n \times (0,\infty)^3)$ such that $v-\phi$ attains its maximum at the point $(x_0,t_0,\lambda_0,\mu_0)$ and $v-\phi=0$ at this point. So $v(x,\frac{t}{k},k\lambda,k\mu)-\phi(x,t,\lambda,\mu)$ attains its maximum at $(x_0,t_0,\lambda_0,\mu_0)$, and $\phi(x_0,kt_0,\frac{\lambda_0}{k},\frac{\mu_0}{k})\geq \phi(x_0,t_0,\lambda_0,\mu_0)$ for any $k>0$. Consider the function $f(k):=\phi(x_0,kt_0,\frac{\lambda_0}{k},\frac{\mu_0}{k})$, thus $f$ attains its minimum at 1, $f^\prime(1)=0$ is what we want for the subsolution test. Similarly, we can check that $v$ is the viscosity supersolution of (\ref{6}).
\end{proof}
Therefore combining (\ref{7})(\ref{6})(\ref{11}), we can deduce that $v$ is the viscosity solution of (\ref{8}). Then we know that  $S_H$ and $S_F$ commute with the same argument.

\section*{Acknowledgements}

I am grateful to Dr. Panrui Ni for offering this idea and all the support. I am also grateful to the University of Science and Technology of China for providing me with the fellowship to visit the University of Tokyo.

\section*{Declarations}

\noindent {\bf Conflict of interest statement:} The author states that there is no conflict of interest.

\medskip

\noindent {\bf Data availability statement:} Data sharing not applicable to this article as no datasets were generated or analysed during the current study.

\bibliography{references}   % 注意不要加 .bib 后缀

\end{document}